\documentclass[a4paper,11pt]{article}
\usepackage[margin=2cm]{geometry}
\usepackage{pict2e}
\usepackage{tikz}
\usepackage[utf8]{inputenc}
\usepackage[english]{babel}
\usepackage{amsthm}
\usepackage{eurosym} 
\usepackage{ textcomp }
 \usepackage[normalem]{ulem}
\usepackage{xcolor,graphicx}
\usepackage{enumerate}

\usepackage{amsmath,amssymb,amsthm,enumitem,dsfont}
\newtheorem{thm}{Theorem}[section]
\newtheorem{lem}[thm]{Lemma}

\newtheorem{proposition}[thm]{Proposition}
\newtheorem{cor}[thm]{Corollary}
\theoremstyle{definition}
\newtheorem{definition}[thm]{Definition}
\newtheorem{remarque}[thm]{Remark}

\newtheorem{exemple}[thm]{Example}

\theoremstyle{remark}
\newtheorem*{rem}{Remark}

\newcounter{case}

\providecommand{\R}{\mathbb R}
\providecommand{\m}{\mathbb}
\providecommand{\xx}{\mathbf x}

\providecommand{\E}{\mathbb{E}}

\providecommand{\prob}{\mathsf{P}}
\providecommand{\PP}{\mathsf{P}}

\providecommand{\N}{\mathbb N}

\providecommand{\Z}{\mathbb Z}
\providecommand{\Id}{\mathrm {Id}}

\providecommand{d}{\mathrm{d}}

\newcommand{\De}{\Delta}
\newcommand{\la}{\lambda}
\newcommand{\ind}{\ensuremath{\mathds 1}}
\newcommand{\de}{\delta}
\newcommand{\be}{\beta}
\newcommand{\ep}{\varepsilon}
\newcommand{\al}{\alpha}
\newcommand{\limla}{\lim_{\lambda\to 0}}
\newcommand{\si}{\sigma}
\newcommand{\ga}{\gamma}
\newcommand{\Ga}{\Gamma}
\definecolor{darkblue}{rgb}{0,0,0.7} 
\newcommand{\darkblue}{\color{darkblue}} 
\newcommand{\defn}[1]{\emph{\darkblue #1}} 
\usepackage{fancyhdr}
\usepackage{hyperref}
\usepackage[normalem]{ulem}
\title{Constant payoff in zero-sum stochastic games}
\author{Olivier Catoni\thanks{CREST, CNRS, ENSAE, Palaiseau, France. olivier.catoni@ensae.fr}, Miquel Oliu-Barton\thanks{CEREMADE, Universit\'e Paris Dauphine, PSL Research Institute, Paris, France. miquel.oliu.barton@normalesup.org} and Bruno Ziliotto\thanks{CEREMADE, CNRS, Universit\'e Paris Dauphine, PSL Research Institute, Paris, France. ziliotto@math.cnrs.fr}}

\begin{document}

\maketitle


\abstract{In a zero-sum stochastic game, at each stage, two adversary players take decisions and receive a stage payoff determined by them and by a controlled random variable representing the state of nature. The total payoff is the normalized discounted sum of the stage payoffs.  In this paper we solve the ``constant payoff'' conjecture formulated by Sorin, Vigeral and Venel (2010): if both players use optimal strategies, then for any $\al>0$, the expected discounted payoff between stage 1 and stage $\al/\la$ tends to the limit discounted value of the game, as the discount rate $\la$ goes to $0$. 
}
\abstract{
Dans un jeu stochastique à somme nulle, à chaque étape, deux joueurs adversaires prennent des décisions et reçoivent un paiement d'étape déterminé par ces décisions, ainsi que par une variable aléatoire contrôlée qui représente l'état de la nature. Le paiement total est la somme escomptée et normalisée  des paiements d'étape. Dans cet article, nous résolvons la conjecture du ``paiement constant", formulée par Sorin, Vigeral et Venel (2010): si les deux joueurs jouent des stratégies optimales, alors pour tout $\al>0$, l'espérance du paiement escompté entre les étapes 1 et $\al/\la$ tend vers la limite de la valeur escomptée du jeu, lorsque le facteur d'escompte $\la$ tend vers $0$.
}
\section{Introduction}

Stochastic games were introduced by Shapley \cite{shapley53} in order to model a repeated interaction between two opponent players in a changing environment.  At each stage $m\in \N$ of the game, players play a zero-sum game that depends on a state variable. Formally, knowing the current state $k_m$, Player 1 chooses an action $i_m$ and Player 2 chooses an action $j_m$. Their choices occur independently and simultaneously and have two consequences: first, they produce a stage payoff $g_m:=g(k_m,i_m,j_m)$ and second, 
they determine the law $q(k_m,i_m,j_m)$ of the next period's state $k_{m+1}$. Thus, the sequence of states follows a Markov chain controlled by the actions of both players.
To any discount rate $\lambda \in (0,1]$ and any initial state $k$ corresponds a $\lambda$-discounted game, denoted by $\Ga_\la(k)$,  in which Player 1 maximizes 
the expectation of $\sum_{m \geq 1} \lambda(1-\lambda)^{m-1} g_m$ given that $k_1=k$, while Player 2 minimizes this same amount. 
 A crucial aspect in this model is that \emph{the current state is commonly observed} by the players at every stage. Another one is stationarity: the transition function and stage payoff function do not change over time. 
 We assume like in Shapley's seminal work, that the set of states and the action sets for both players are \emph{finite}.

%
 Shapley \cite{shapley53} proved that for any initial state $k$ and any discount rate $\la$, the $\la$-discounted stochastic game
has a value $v_\la(k)$, which is the unique fixed point of a contracting map. 
 Furthermore, both players have optimal strategies that depend on the past only through the current state.  
 A wide area of research is to investigate the properties of the discounted game, when $\lambda$ tends to 0. Intuitively, this corresponds to a game played either between very patient players, or between players who are very likely to interact a great number of times. Building on Shapley's results, Bewley and Kohlberg \cite{BK76} proved that $v_\la(k)$ converges as $\lambda$ tends to 0.  An alternative proof of this result was recently obtained by Oliu-Barton \cite{OB14}, using probabilistic and linear programming techniques. Mertens and Neyman \cite{MN81} proved then the existence of the so-called \textit{uniform value} $v^*(k)$, and its equality with $\limla v_\la(k)$. A new characterization for $v_\la(k)$, and a formula for $v^*(k)$ were recently obtained by Attia and Oliu-Barton \cite{AOB18}. Efficient algorithms to compute these values were then deduced by  
 Oliu-Barton \cite{OB20}. 
The finiteness of the state space plays a crucial role in these results, as highlighted by the counterexamples of Vigeral \cite{vigeral12} and Ziliotto \cite{ziliotto13} who considered, respectively, the case of compact action sets and compact state space. 
 
A remarkable property, referred to as the \emph{constant payoff property} was proved by Sorin, Venel and Vigeral \cite{SVV10} in the framework of single decision-maker problems:  
for any sufficiently small $\lambda$ there exists an optimal strategy so that 
the expectation of the cumulated payoff  $\sum_{m=1}^M \lambda(1-\lambda)^{m-1} g_m$ given $k_1=k$ is approximately equal to $(\sum_{m=1}^M\lambda(1-\lambda)^{m-1})v^*(k)$. 
Note that the positive weights $\la(1-\la)^{m-1}$ add up to $1$, so that $\sum_{m=1}^M \lambda(1-\lambda)^{m-1}$ represents the fraction of the game that has already been played at stage $M$. 
The constant payoff property holds as soon as the discounted value converges as the discount rate goes to $0$, and that the convergence is uniform in the state space. Further, it was conjectured that under similar conditions the constant payoff property should hold for any class of two-player zero-sum stochastic games that satisfy the same assumptions. As the discounted value of finite stochastic games converges, and the convergence is uniform (by finiteness), the conjecture directly applies to this class of games. 

The constant payoff property is not straightforward. Lehrer and Sorin \cite{LS93} provided a simple example of a Markov decision problem over a countable set of states where this property fails: when the decision-maker plays an optimal strategy, he gets a high payoff during the first half of the game, and a low payoff during the second half.  
However, it was known to hold for finite \emph{absorbing games}, a subset of stochastic games in which all states except one are absorbing. 
 Beyond the finite framework, Sorin and Vigeral \cite{SV18} established the constant payoff property for absorbing games with compact action sets and jointly continuous payoff and transition functions. Oliu-Barton \cite{OB17} established the same property for the splitting game, a stochastic game with compact action sets and jointly continuous payoff and transition functions introduced by Sorin \cite{sorin02} to capture the information transmission in repeated games with incomplete information. 
 Let us note that, in spite of their differences, absorbing games and games with incomplete information have in common that the dynamics of the game has an irreversible property, which is not present in stochastic games.\\ 

The main contribution of this paper is to establish that finite stochastic games have the constant payoff property, and thus to solve the conjecture in Sorin, Venel and Vigeral \cite{SVV10}. Moreover, a property more general than the conjecture is established (\textit{strong constant payoff property}): 
for any sufficiently small $\lambda$, \textit{for any pair of optimal strategies}, 
the expectation of the cumulated payoff  $\sum_{m=1}^M \lambda(1-\lambda)^{m-1} g_m$ given $k_1=k$ is approximately equal to $(\sum_{m=1}^M\lambda(1-\lambda)^{m-1})v^*(k)$.
\\
The proof relies heavily on the semi-algebraic approach developed by Bewley and Kohlberg \cite{BK76}, namely that the value function $v_\lambda(k)$ and a family of optimal stationary strategies $(x^1_\lambda,x^2_\lambda)$ admit a Puiseux series expansion in a neighborhood of $0$ (\textit{optimal Puiseux strategy profiles}). It is decomposed in two parts. First, we establish that the constant payoff property holds for optimal Puiseux strategy profiles (\textit{weak constant payoff property}). This readily proves the constant payoff conjecture for finite stochastic games. In a second part, we generalize this property to any family of optimal strategies, to obtain the strong constant payoff property. 

 The remainder of the paper is divided as follows. Section \ref{main} presents the model and main result. Section \ref{section_WP} is devoted to the proof of the weak constant payoff property. Section \ref{section_SP} proves the strong constant payoff property.  Section \ref{section_annexe} gives some examples and remarks. 

\section{Model and main results} \label{main}
\subsection{Stochastic games}
We consider throughout this paper a standard two-player zero-sum stochastic game, as introduced by Shapley \cite{shapley53}. Such games are described by a 5-tuple $\Ga=(K,I,J,g,q)$, where $K$ is the set of states, $I$ and $J$ are the action sets of Player 1 and 2 respectively, $g:K\times I\times J\to \R$ is the payoff function and $q:K\times I\times J\to \De(K)$ is the transition function, where for each finite or countable set $X$, we denote by $\De(X)$ the set of probability distributions over $X$. \textbf{We assume that $K$, $I$ and $J$ are finite sets}.  


\paragraph{Outline of the game.}
The game $\Ga$ proceeds as follows:  at every stage $m\geq 1$, knowing the current state $k_m$, the players choose actions $i_m$ and $j_m$ independently; Player 1 receives the stage payoff $g(k_m,i_m,j_m)$, and Player 2 receives $-g(k_m,i_m,j_m)$. A new state $k_{m+1}$ is drawn according to the probability $q(\cdot\,|\, k_m, i_m,j_m)$. The players observe the pair of actions $(i_m,j_m)$ and the new state $k_{m+1}$, and the game proceeds to stage $m+1$. 


\paragraph{Discounted stochastic games.} For any discount rate $\la \in(0,1]$, we denote by $\Ga_\la(k)$ the stochastic game  $\Ga$ with initial state $k$ and where Player 1 maximizes, in expectation, the normalized $\la$-discounted sum of stage payoffs 
\[
\sum\nolimits_{m\geq 1}\la(1-\la)^{m-1} g(k_m,i_m,j_m),
\]
while Player 2 minimizes this amount. 
More precisely, we consider the 
case where the strategies of the two players form the saddle point 
of a $\min$-$\max$ problem, as explained below. 


\paragraph{Strategies.}
The sequence $(k_1,i_1,j_1,...,k_m,i_m,j_m,...)$ generated along the game is called \textit{a play}. The set of plays is $(K\times I\times J)^{\N}$.  
\begin{enumerate}
\item[$(i)$]
A \defn{strategy} for a player specifies a mixed action to each possible set of past observations: formally, a strategy for Player 1 is a collection of maps $\sigma^1=(\sigma^1)_{m \geq 1}$, where $\sigma^1_m:(K \times I \times J)^{m-1} \times K \rightarrow \Delta(I)$. Similarly, a strategy for Player 2 is a collection of maps $\sigma^2=(\sigma^2)_{m \geq 1}$, where $\sigma^2_m:(K \times I \times J)^{m-1} \times K \rightarrow \Delta(J)$. 
\item[$(ii)$]
A \defn{stationary strategy} plays according to the current state only. Formally, a stationary strategy for Player 1 is a mapping $x^1:K\to \De(I)$. Similarly, a stationary strategy for Player $2$ is a mapping $x^2:K\to \De(J)$. 
\item[$(iii)$] A \defn{strategy profile} is a pair of strategies $(\si^1,\si^2)$.
\end{enumerate}
\noindent \textbf{Notation.} The sets of strategies for Player 1 and 2 are denoted by $\Sigma^1$ and $\Sigma^2$, respectively, and the sets of stationary strategies by $\De(I)^K$ and $\De(J)^K$.\\ 


We denote by $\PP^{k}_{\si^1,\si^2}$ the unique probability measure on the set of plays $(K\times I\times J)^\N$ such that, for any finite play $h^n=(k_1,i_1,j_1,\dots,k_{n-1},i_{n-1},j_{n-1},k_n)$ one has
\begin{eqnarray*}\prob^k_{\si^1,\si^2}(h_n) =\prod_{m=1}^{n-1} \si^1_m[h^n_m](i_m) \si^2_m[h^n_m](j_m)q(k_{m+1}|k_m,i_m,j_m) \, .\end{eqnarray*} 
 where $h^n_m$ is the restriction of $h^n$ to the first $m$ stages, i.e. $h^n_1:=k_1$ and for all $2\leq m \leq n$: 
 $$ h^n_m:=(k_1,i_1,j_1,\dots, k_{m-1},i_{m-1},j_{m-1},k_m)\, .$$
The extension to infinite plays follows from the Kolmogorov extension theorem. Thus, $\PP^{k}_{\si^1,\si^2}$ is the unique probability measure on plays induced by the pair $(\si^1,\si^2)$ in the stochastic game starting from state $k$ (note that the dependence on the transition function $q$ is omitted). The expectation with respect to the probability $\PP^k_{\si^1,\si^2}$ is denoted by $\E^k_{\si^1,\si^2}$. 
For any $\la\in (0,1]$ and any $k\in K$, 
we denote by $\ga_\la(k,\, \cdot\,,\, \cdot\,): \Sigma^1\times \Sigma^2\to \R$ the payoff function corresponding to the game $\Ga_\la(k)$: 
\begin{equation}\label{expected_payoff}
\ga_\la(k,\si^1,\si^2):=\E_{\si^1,\si^2}^{k}\left[\sum\nolimits_{m\geq 1}\la(1-\la)^{m-1} g(k_m,i_m,j_m)\right].\end{equation}
 \paragraph{Shapley's results.} For any discount rate $\la\in (0,1]$ and any initial state $k\in K$, Shapley \cite{shapley53} proved that the discounted stochastic game $\Ga_\la(k)$ has a value, so that the following equalities hold: 
\begin{equation}\label{valSG}
v_{\lambda}(k)=\max_{\sigma^1 \in \Sigma^1} \min_{\sigma^2 \in \Sigma^2} \gamma_{\lambda}(k,\sigma^1,\sigma^2)=\min_{\sigma^2 \in \Sigma^2} \max_{\sigma^1 \in \Sigma^1} \gamma_{\lambda}(k,\sigma^1,\sigma^2)\,.
\end{equation}
Furthermore, both players have optimal stationary strategies in $\Ga_\la$, where a strategy $\sigma^1 \in \Sigma^1$ is optimal for Player 1 if for any $k \in K$, it realizes the maximum in the left-hand side of \eqref{valSG}, and a strategy $\sigma^2 \in \Sigma^2$ is optimal for Player 2 if for any $k \in K$, it realizes the minimum in the right-hand side of  \eqref{valSG}.\\

\noindent \textbf{Notation.} The set of optimal strategies for Player 1 and 2 in the game $\Ga_\la$ are denoted by $\Sigma^1_\la$ and  $\Sigma^2_\la$, respectively. 

\subsubsection{Puiseux strategies} 
A map $f:(a,b) \to \R$ is a \defn{Puiseux series} on $(a_0,b_0)\subset (a,b)$ if there exists $m_0\in \Z$, $N\in \N$ and a real sequence $(c_m)_{m\geq 0}$ so that 
$$f(\la)=\sum_{m\geq m_0} c_m \la^{m/N}\qquad \forall \la\in (a_0,b_0).$$
A function $f:(0,1] \to \R$ admits a Puiseux expansion at $0$ if there exists $\la_0$ so that $f$ is a Puiseux series on $(0,\la_0)$. Clearly, if $f$ is bounded then one can take $m_0=0$. 

\begin{definition} A \defn{Puiseux strategy profile} \label{def_Puiseux} is a family of stationary strategy profiles $(x^1_\la,x^2_\la)_{\la\in(0,1]}$ so that, for some $\la_0 \in (0,1]$, the mappings $\la\mapsto x^1_\la(k,i)$ and $\la\mapsto x^2_\la(k,j)$ are bounded real Puiseux series on $(0,\la_0)$, for all $(k,i,j)\in K\times I\times J$.
\end{definition}
\begin{definition} An \defn{optimal Puiseux strategy profile} is a Puiseux strategy profile $(x^1_\la,x^2_\la)_{\la\in(0,1]}$ so that for all $\la\in(0,\la_0]$ and $k \in K$, the stationary strategies $x^1_\la$ and $x^2_\la$ are optimal in $\Gamma_\la$. 
\end{definition}

 \subsubsection{The semi-algebraic approach} \label{semi_alg}
Fix $\la\in(0,1]$ and let $v_\la\in \R^K$ be the vector of values. 
Building on Shapley's results \cite{shapley53}, Bewley and Kohlberg \cite{BK76} defined a subset $S\subset \R \times \R^K\times \R^{K\times I}\times \R^{K\times J}$ by setting 
$$(\la, v, x^1,x^2)\in S\ \Longleftrightarrow \ \begin{cases} \la\in \R  \text{ is a discount rate}\\
v\in \R^K \text{ is the vector of values of } \Ga_\la\\
(x^1, x^2)\in \R^{K\times I}\times \R^{K\times J} \text{ is a pair of optimal stationary strategies in } \Gamma_\la\, .\end{cases}$$

The set $S$ is semi-algebraic, as it can be described by the following finite set of polynomial equalities and inequalities: \begin{eqnarray*}
0< \la &\leq & 1\\
\forall (k,i), \ x^1(k,i)\geq 0,\ \text{and } \ \forall k, \quad \sum\nolimits_{i\in I} x^1(k,i)&=&1\\
\forall (k,j), \ x^2(k,j)\geq 0,\ \text{and } \ \forall k, \quad \sum\nolimits_{j\in J} x^2(k,j)&=&1\\
\forall (k,j),\quad \sum\nolimits_{i\in I} x^1(k,i)\left(\la g(k,i,j)+(1-\la)\sum\nolimits_{\ell \in K}q(\ell|k,i,j)v(\ell)\right) & \geq & v(k)\\
\forall (k,i),\quad \sum\nolimits_{j \in J} x^2(k,j)\left(\la g(k,i,j)+(1-\la)\sum\nolimits_{\ell \in K}q(\ell | k,i,j)v(\ell)\right) & \leq & v(k)\,.
\end{eqnarray*}
By the Tarski-Seidenberg elimination theorem, the functions $\la\mapsto v_\la(k)$ are real, semi-algebraic functions, for each initial state $k\in K$. Similarly, there exist a selection of optimal strategies $(x^1_\la,x^2_\la)$ such that the maps $\la\mapsto x^1_\la(k,i)$ and $\la\mapsto x^2_\la(k,j)$ are real semi-algebraic functions as well, for all $(k,i,j)$. By the Puiseux theorem, any real semi-algebraic function $f:(0,1]\to \R$ admits a Puiseux expansion in some neighborhood of $0$. Hence,
\begin{itemize}
\item For each $k\in K$, the map $\la\mapsto v_\la(k)$ admits a Puiseux expansion at $0$, so that the limit $v^*(k):=\lim_{\la\to 0}v_\la(k)$ exists. 
\item There exists an optimal Puiseux strategy profile $(x^1_\la,x^2_\la)_{\la\in(0,1]}$. 
\end{itemize}







\subsubsection{The game on $[0,1]$}\label{clock}
Let $\la\in(0,1]$. 
For any $M\in \N$, define the following map:
 \[\eta(\la,\,\cdot\,):\N\to [0,1],\qquad 
\eta(\la,M):=\sum_{m=1}^M \la(1-\la)^{m-1}\,.\] 
It can be interpreted as a {clock} that indicates the \emph{fraction of the game} that has already been played after any given number of stages.  
Conversely, to any fraction of the game $t\in[0,1]$ corresponds a stage where the sum of weights of the previous stages is approximately equal to $t$. Formally, we introduce the \emph{inverse-clock} map by 
\[\varphi(\la,\,\cdot\,):[0,1]\to \N \cup \left\{+\infty\right\}, \qquad 
\varphi(\la,t):=\inf \{M\geq 1,\ \eta(\la,M)\geq t\}  \ =\left\lceil\frac{\ln(1-t)}{\ln(1-\lambda)}\right\rceil   \, , \]  
where $\lceil x \rceil$ denotes the upper integer part of $x$.
The notion of clock and inverse-clock, which are now standard, were initiated by Sorin \cite{sorin03}, and allow to consider the discrete-time game $\Gamma_\la(k)$ as a game played on the time interval $[0,1]$. 
\paragraph{Cumulated payoffs.} For any fraction $t\in [0,1]$ we extend the definition of the payoff function to the map $\ga_\la(k,\, \cdot\,,\, \cdot\,,\, \cdot\,):  \Sigma^1\times \Sigma^2 \times [0,1] \to \R$ by setting 
\begin{eqnarray}\label{cumul_t}
\ga_\la(k,\si^1,\si^2;t)&:=&\E_{\si^1,\si^2}^{k}\left[\sum\nolimits_{m= 1}^{\varphi(\la,t)} \la(1-\la)^{m-1} g(k_m,i_m,j_m)\right].
\end{eqnarray}
For convenience, for any pair of strategies $(\si^1,\si^2)$ we set $\ga_\la(\si^1,\si^2;t)\in \R^K$ to be the vector of payoffs $\ga_\la(k,\si^1,\si^2;t)$, $k\in K$. Note also that $\ga_\la(\si^1,\si^2)=\ga_\la(\si^1,\si^2;1)$ by definition. 


\subsection{Main result} 
Our main result is a precise characterisation of the cumulated payoff at time $t$, when both players use optimal strategies in the game $\Ga_\la(k)$, for sufficiently small $\la\in(0,1]$. 

\begin{thm}[Strong constant payoff property] \label{paiement_constant_SP} 
For any $\ep>0$, there exists $\lambda_0 \in (0,1)$ so that for all $\lambda \in (0,\lambda_0)$, $t \in [0,1]$,  $k \in K$, and $(\sigma^1, \sigma^2) \in \Sigma_\la^1 \times \Sigma_\la^2$ one has: 
\begin{equation}\label{strong}
\left|\ga_\la(k,\si^1,\si^2;t)-t v^*(k)\right| \leq \ep \, .
\end{equation}
\end{thm}
This result solves the conjecture raised by Sorin, Venel and Vigeral \cite{SVV10} in a \emph{strong sense}. That is, where \cite{SVV10} conjectured the existence of a strategy profile $(\si^1,\si^2)$ so that \eqref{strong} holds, we prove that the constant payoff property holds for \emph{every} optimal strategy profile. 
\section{Weak constant payoff property}\label{section_WP}
 

\begin{thm}[Weak constant payoff property]\label{paiement_constant_WP} 
For any optimal Puiseux strategy profile \linebreak 
$(x^1_\la,x^2_\la)_{\la\in(0,1]}$,
$$\limla \ga_\la(x^1_\la,x^2_\la; t)= tv^*\qquad \forall t\in [0,1]\, .$$ 
\end{thm}
\begin{rem}Theorem \ref{paiement_constant_WP} establishes the constant payoff conjecture of Sorin, Venel and Vigeral~\cite{SVV10} for a specific family of strategy profiles. 
\end{rem}
\begin{rem}
For $\ep>0$, let $(\sigma_\ep, \tau_\ep$) be a pair of $\ep$-optimal \emph{uniform} strategies, that is, satisfying $\ga_\la(k,\si_\ep,\tau)\geq  v^*(k)-\ep$ and $\ga_\la(k,\si,\tau_\ep)\leq  v^*(k)+\ep$ for all $k \in K$, for all pair of strategies $(\si,\tau)$ and all $\la$ small enough. Such a pair exists by Mertens and Neyman \cite{MN81}. Then, for all $k \in K$, 
for any sequence of strategies $(\sigma_\la,\tau_\la)$ and any $t \in [0,1]$ one has $\liminf_{\la\to 0} \ga_\la(k,\sigma_\ep,\tau_\la; t) \geq tv^*(k)-\ep$ and 
$\limsup_{\la\to 0} \ga_\la(k,\sigma_\la,\tau_\ep; t) \leq tv^*(k)+\ep$. In particular, $(\sigma_\ep,\tau_\ep)$ satisfies the weak constant payoff property, up to an error term $\ep$. Nonetheless, these strategies are in general not stationary. 
\end{rem}
\begin{proof} 

 In the sequel, $(x^1_\la,x^2_\la)$ denotes an optimal Puiseux strategy profile. 
Let $\la_0>0$ be such that all the coordinates of $\la\mapsto x^1_\la$ and $\la\mapsto x^2_\la$ are Puiseux series on $(0,\la_0)$.  The result is clear for $t=0$ and $t=1$ so we fix in the sequel some $t\in (0,1)$. 
 
\paragraph{Step 0:} \emph{Introduction of tools}. \\
For any stationary strategy profile $(x^1,x^2)$, define the matrix $\Pi(\la,x^1,x^2)\in \R^{K\times K}$ for all $\la \in (0,1]$ and  the vector $g(x^1,x^2)\in \R^{K}$ by setting 
\begin{eqnarray*} \Pi^{k,\ell}(\la,x^1,x^2)&:=&\E_{x^1,x^2}^{k}\left[ \sum\nolimits_{m\geq 1}\la (1-\la)^{m-1}
\ind_{\{ k_m=\ell\} } \right]\qquad \forall (k,\ell)\in K^2, 
\\ g^k(x^1,x^2)&:= &\sum_{(i,j)\in I\times J}x^1(k,i)x^2(k,j)g(k,i,j)\qquad \forall k \in K\,.
\end{eqnarray*} 
The real $\Pi^{k,\ell}(\lambda,x^1,x^2)$ represents the expected (discounted) fraction of the game spent in state $\ell$, given that players play stationary strategies $x^1$ and $x^2$, and the initial state is $k$. 
The real $g^k(x^1,x^2)$ represents the expected stage payoff, given that players play $x^1$ and $x^2$ and the initial state is $k$. 
We claim that $\la\mapsto \Pi(\la,x^1_\la,x^2_\la)$ is a bounded real Puiseux series, so that the limit $\Pi:=\limla \Pi(\la,x^1_\la,x^2_\la)\in \R^{K\times K}$ exists. 
Indeed, 
define a stochastic matrix $Q_\la\in \R^{K\times K}$ 
\begin{equation}
\label{def_Q_lambda}
Q_\la(k,\ell) := \sum_{(i,j)\in I\times J} x_\la^1(k,i)x_\la^2(k,j)q(\ell\,|\,k,i,j)\qquad \forall (k,\ell)\in K^2, 
\end{equation}
so that 
\[
\Pi(\la,x^1_\la,x^2_\la) = \sum_{m\geq 0}\la(1-\la)^m 
Q_\la^m.
\]
The real $Q_\lambda(k,\ell)$ represents the probability that tomorrow's state is $\ell$, given that the state is $k$ today and players play $(x^1_\la,x^2_\la)$.
Consider the Markov chain $M_\la$ on $K\cup \{*\}$ defined as follows:
$$M_\la(k,\ell)=\begin{cases} (1-\la)Q_\la(k,\ell) & \text{ if } k,\ell \in K\\
\la & \text{ if } k \in K, \ \ell = *\\
(1+|K|)^{-1} & \text{ if } k=*, \ \ell \in K \cup \left\{*\right\}.\end{cases}$$
Remark that, if $X_n$ is the Markov chain with transitions $M_\la$, then 
\begin{eqnarray*}\sum\nolimits_{m\geq 0}(1-\la)^{m}Q_\la^m(k,\ell)&=&\E\left[ \sum\nolimits_{m=1}^{\tau(K)} \ind_{\{X_m=\ell\}}\,|\, X_0=k\right],\\
&=& \frac{\sum_{\pi \in G_{k,\ell}(K\backslash \{\ell\})} \prod_{(k',\ell')\in \pi} M_\la(k',\ell')}
{\sum_{\pi \in G(K)} \prod_{(k',\ell')\in \pi} M_\la(k',\ell')}\, , \end{eqnarray*}
where for any set $A$, $\tau(A):=\inf\{m\geq 0, X_m\notin A\}$, $G(A)$ is the set of acyclic graphs such that exactly one arrow starts from any point of $A$ and no arrow starts outside of $A$, and $G_{k,\ell}(A)$ is the set of graphs of $G(A)$ such that $k$ leads to $\ell$, where $k\in A$ and $\ell\notin A$ (see \cite[Lemma 3.1]{catoni99} for a proof). 
We conclude that $\Pi(\lambda, x^1_\la, x^2_\la)$ is a Puiseux series since
it is the ratio of two finite sums of Puiseux series.

\paragraph{Step 1:}  \emph{The equality $\Pi v^*=v^*$}.\\
Define the map $f: (0,\la_0)^3\to \R^K$ by setting 
\begin{equation*}
f(\lambda,\lambda^1,\lambda^2)= \Pi(\lambda,x_{\lambda^1}^1,x_{\lambda^2}^2)  g(x_{\la^1}^1,x_{\la^2}^2)\qquad \forall (\la,\la^1,\la^2)\in (0,\la_0)^3\,.
\end{equation*}
Note  that $f$ is differentiable on $(0,\la_0)^3$, because it is a power series in the variables $\la$, $(\la^1)^{1/N}$ and $(\la^2)^{1/N}$, for some $N\in \N$.
 For each $k\in K$,
$x^1_\lambda$ and $x^2_\lambda$ are optimal strategies in $\Gamma_\lambda(k)$, so that 
the map $(\lambda^1,\lambda^2)\rightarrow f^k(\lambda,\lambda^1,\lambda^2)$ has a saddle point at $(\lambda,\lambda)$, for each $\la\in(0,\la_0)$. 
Hence, 
its partial derivatives satisfy
\begin{equation}\label{eqderiv0}
\frac{\partial f}{\partial \lambda^1}(\lambda,\lambda,\lambda)=\frac{\partial f}{\partial \lambda^2}(\lambda,\lambda,\lambda)=0\, .
\end{equation}
For any $\la\in (0,1]$, set $h(\lambda):=f(\lambda,\lambda,\lambda)\in \R^K$. By the choice of $(x_\la^1,x^2_\la)$, $h(\la)=v_\lambda$. 
Define 
a payoff vector $g_\la\in \R^K$ by: 
\[
g_\la(k) := g^k(x_\la^1,x_\la^2)= \sum_{(i,j)\in I\times J} x_\la^1(k,i)x_\la^2(k,j)g(k,i,j)\qquad \forall k\in K\,.
\]
The relation \eqref{eqderiv0} implies that the derivative of $h$ satisfies
\begin{eqnarray*}
h'(\lambda)&=& \left(\frac{\partial}{\partial \la}+\frac{\partial}{\partial\la^1}+\frac{\partial}{\partial\la^2}\right) f(\la,\la^1,\la^2)_{|\la^1=\la^2=\la} \\
&=&  \frac{\partial}{\partial \la}f(\la,\la^1,\la^2)_{|\la^1=\la^2=\la} \\
&=& \left[\frac{\partial}{\partial \la}  \Pi(\lambda,x_{\lambda^1}^1,x_{\lambda^2}^2)_{|\la^1=\la^2=\la}\right]  g_\lambda
\\
&=& \sum_{m \geq 0} (1-\lambda)^{m} Q^{m}_{\lambda}g_\la
-\sum_{m \geq 0} m \lambda(1-\lambda)^{m-1} Q^{m}_{\lambda}g_\la\,.
\end{eqnarray*}
As $\Pi(\la,x^1_\la,x^2_\la) \, g_\la =\sum_{m\geq 0}\la(1-\la)^m Q_\la^m g_\la= v_\la$, it follows that
\begin{eqnarray*}
\Pi(\la,x^1_\la,x^2_\la) \, v_\lambda&=& 
\Pi(\la,x^1_\la,x^2_\la) \Pi(\la,x^1_\la,x^2_\la) \, g_\la \\&=&
\sum_{m \geq 0, n \geq 0} \lambda^2 (1-\lambda)^{n+m}
Q_\lambda^{m+n} g_\la \\&=&\sum_{m \geq 0} (m+1) \lambda^2(1-\lambda)^mQ^m_\lambda g_\la
\\
&=& \la v_\la +\la(1-\la)\sum_{m \geq 0} m \lambda(1-\lambda)^{m-1} Q^{m}_{\lambda}g_\la,
\end{eqnarray*}
where $\la^2$ stands for ``$\la$ square'' in the two previous equations. Consequently, replacing the expression of $h'(\la)$ one obtains
\begin{eqnarray*}
\Pi(\la,x^1_\la,x^2_\la) v_\lambda&=&\lambda v_\lambda+\lambda(1-\lambda)\left(\lambda^{-1} v_\lambda-h'(\lambda)\right)
\\
&=& v_\lambda-\lambda(1-\lambda) h'(\lambda)\,.
\end{eqnarray*}
Since each coordinate of $h(\la)$ is a bounded real Puiseux series, it follows that 
$\lim_{\lambda \rightarrow 0} \lambda h'(\lambda)=0$. Taking $\la$ to $0$ in the previous expression thus gives $\Pi  v^*=v^*$.

\paragraph{Step 2:} \emph{Relation between $\Pi$ and the occupation measure at time $t$}.\\ 
For every $\la\in(0,1]$, by the Markov property, 
\begin{equation} \label{prog_dyn}
\Pi(\la,x^1_\la,x^2_\la)=\sum_{m= 1}^{\varphi(\la,t)} \la(1-\la)^{m-1}Q_\la^{m-1}+(1-\lambda)^{\varphi(\lambda,t)} Q^{\varphi(\lambda,t)}_{\lambda}\Pi(\la,x^1_\la,x^2_\la)\,.
\end{equation}
In order to establish Theorem \ref{paiement_constant_WP}, we are going to prove that $t v^*$ is the only accumulation point of $(\gamma_\la(x^1_\la,x^2_\la;t))$, as $\la$ vanishes. 
Let $\gamma^* \in \m{R}^K$ be such an accumulation point, and let $g^*:=\limla g_\la\in \R^K$. By consecutive extractions, 
one can find a vanishing sequence $(\lambda_r)$ such that $(\gamma_{\la_r}(x^1_{\la_r},x^2_{\la_r};t))$ converges to $\gamma^*$, $Q_{\lambda_r}^{\varphi(\lambda_r,t)}$ converges to some $\pi_t \in \R^{K\times K}$, and $\sum_{m= 1}^{\varphi(\la_r,t)} \la_r(1-\la_r)^{m-1}Q_{\la_r}^{m-1}$ converges to some 
$\Pi_t \in \R^{K\times K}$. In particular, $\gamma^*=\Pi_t g^*$, and thus our aim is to prove that $\Pi_t g^*=tv^*$. 

Setting $\la=\lambda_r$ and having $r$ 
going to infinity in \ref{prog_dyn}, we obtain
\begin{equation}\label{rec2}
\Pi=\Pi_t +(1-t)\pi_t \Pi\,. 
\end{equation}
Iterating this equation, one gets 
\begin{equation}\label{rec}
\Pi=\sum_{m \geq 0} (1-t)^{m} \pi_t^{m} \Pi_t\,.
\end{equation}
Set $P_t:=\frac{1}{t}\Pi_t$, which is a stochastic matrix on the state space $K$, and note that the previous relation can be expressed as 
\begin{equation}\label{zer}
\Pi=\m{E}(\pi_t^{X})P_t,
\end{equation}
where $X+1$ is a geometric random variable with parameter $t$, i.e. $\m{P}(X=m)=t(1-t)^{m}$ for all $m\geq 0$. %
\paragraph{Step 3:} \emph{A lemma on stochastic matrices}.  \\
Let $X$ be the random variable defined in Step 2.
Since $\m{P}(X = 0 ) > 0$, 
for any stochastic matrix $M$, the stochastic matrix 
$N:=\m{E}(M^{X})$ is aperiodic. 
Therefore $N^n$ has a limit when $n$ goes to infinity, that we call
$N^\infty$. 
We claim that
\begin{equation}\label{rft}
M N^\infty=N^\infty \,.
\end{equation}
Indeed, since $\m{P}(X = 1) > 0$,
$\sup_{n \in \m{N}} M^n(i,j) > 0  \Leftrightarrow 
\sup_{n \in \m{N}} N^n(i,j) > 0$, 
so that the recurrent communicating classes of $M$ and $N$ are the 
same. The number of such classes is equal to the dimension 
of $\ker(M-I)$, 
that is therefore also the dimension of $\ker ( N - I)$. Moreover, as
$Nf = \m{E}(M^X f)$, $Mf = f \Rightarrow Nf = f$, so that 
\[
\ker( M - I ) \subset \ker( N - I). 
\]
Consequently, these two eigenspaces are equal, since they
have the same dimension. Since $(N-I)N^\infty = 0$,
the columns of $N^{\infty}$ belong to $\ker(N-I)$,
and therefore also to $\ker(M-I)$, so that
$(M-I)N^{\infty} = 0$ as claimed.
\paragraph{Step 4:} \emph{The equality $\pi_t v^*=v^*$ for all $t\in(0,1)$}. \\ 
Let 
$P_t^{\infty}$ be an accumulation point of the sequence $(P_t^n)_n$. 
The matrices $\m{E}(\pi_t^{X})$ and $P_t$ commute because they are limits of weighted sums of powers of $Q_\la$, which commute. Hence, using  the equality $\Pi v^*=v^*$ established in Step 1, and the relation \eqref{zer}, it follows that for all $n\in \N$,
\begin{eqnarray*}
 v^*&=&\Pi^n v^*\\ 
 &=&\left(\m{E}(\pi_t^{X})P_t \right)^n v^*
\\
&=& 
\m{E}(\pi_t^{X})^n P_t^n  v^*\,.
\end{eqnarray*}
Thus, as $n$ tends to infinity along a subsequence defining $P_t^\infty$, one has
\begin{equation}\label{rft2}
 v^*=\m{E}(\pi_t^{X})^{\infty} P_t^{\infty}  v^*\,.
\end{equation}
Combining the equality \eqref{rft} of Step 3 with $M=\pi_t$ and $N=\m{E}(\pi_t^{X})$, and \eqref{rft2}, one obtains 
\begin{eqnarray*}
\pi_t v^*&=&\pi_t \m{E}(\pi_t^{X})^{\infty} P_t^{\infty} v^*\\ &=&\m{E}(\pi_t^{X})^{\infty} P_t^{\infty} v^*\\ &=&v^*\,.\end{eqnarray*}
\paragraph{Step 5:} \emph{Conclusion: $\Pi_t g^* = tv^*$ for all $t\in[0,1]$}. \\
Multiplying the two sides of \eqref{rec2} by $g^*$ yields 
$$\Pi g^*=\Pi_t g^*+(1-t)\pi_t \Pi g^*\,.$$
Yet, for all $\la\in(0,1]$ one has $\Pi(\la,x^1_\la,x^2_\la) g_\la=v_\la$. Taking limits as $\la$ goes to $0$, it follows that $\Pi g^*=v^*$. 
Combined with the equality $\pi_t v^*=v^*$ obtained in Step 4, this gives $v^*=\Pi_t g^* +(1-t) v^*$, so that $\Pi_t g^* = tv^*$.\end{proof} 


\begin{rem} In Step 1 we obtained the following expression for the derivative of $v_\la$:
$$\frac{\partial}{\partial \la} v_\la= \frac{1}{\la(1-\la)}\left(v_\la-\Pi(\la,x^1_\la,x^2_\la) v_\la\right)\,.$$

\end{rem}

\section{Strong constant payoff property}\label{section_SP}
We now prove our main result: Theorem \ref{paiement_constant_SP}.  Roughly speaking, we want to prove that the constant payoff property, which is true for any optimal Puiseux strategy profile, holds for \emph{any} pair of optimal strategies. The main idea is the following: an equivalence between the strong constant payoff property and the convergence to $0$ of the values of a certain class of discounted Markov decision processes. 
We start with a technical property for real sequences, from which we derive an equivalent formulation of the strong constant payoff property. \\


For each $k \in K$, define $X^1_{\lambda}(k)\subset \De(I)$ (resp., $X^2_{\lambda}(k)\subset \De(J)$) the set of optimal strategies for Player $1$ (resp., 2) in the one-shot zero-sum game with action sets $I$ and $J$ and payoff:
\begin{equation}\label{shap_game}R(i,j):=\la g(k,i,j)+(1-\la)\sum_{\ell \in K}q(\ell|k,i,j)v_\la(\ell)\, .\end{equation}

The following lemma is a direct consequence of \cite[Corollary 2.6.3]{RenaultTSE}. We state it for Player 1 but, as players have symmetric roles, a similar result holds for Player 2. 
 \begin{lem} 
 A general strategy $\si^1$ of Player $1$ is optimal in the discounted stochastic game $\Ga_\la$ if, and only if, for any $k_1 \in K$, for any $m \geq 1$, for any strategy $\sigma^2 \in \Sigma^2$ of Player $2$ and any finite history $h^m \in H_m$ such that $\m{P}^{k_1}_{\sigma^1,\sigma^2}(h^m)>0$, 
 Player $1$ plays a mixed action in $X^1_\la(k_m)$. 
 \end{lem}


\subsection[Characterisation]{Characterisation of the strong constant payoff property}
 \subsubsection{A technical lemma on real sequences} \label{real_seq}

 



Let $(u^{\lambda}_m)_{m \geq 1}$ be a fixed family of real sequences so that, for some constant $C\geq 0$, for all $\la \in (0,1]$ and all  $m\geq 1$, 
\begin{equation} \label{condseq}
\left|u^{\lambda}_{m+1}-u^{\lambda}_m \right| \leq C \lambda \quad \text{and} \quad \left|u^{\lambda}_m\right| \leq C\, .  
\end{equation}
For each $\de>0$ we set 
\begin{equation*}
B_\lambda(\delta):=\sum\nolimits_{m \geq 1} \de \la (1-\de \la)^{m-1} u^{\lambda}_m \, .
\end{equation*}
\begin{proposition}\label{analyse}
The two following statements are equivalent:
\begin{enumerate}
\item[$(i)$] 
For all $t \in (0,1)$, 
$u^{\la}_{\varphi(\la,t)}$ vanishes as $\lambda$ tends to 0.
\item[$(ii)$]  
For all $\delta>0$, $B_\lambda(\delta)$ vanishes as $\lambda$ tends to 0.
\end{enumerate}
\end{proposition}
\begin{proof} 
Consider the functions 
\[
f_{\lambda}(x) = u_{\lfloor x / \la \rfloor + 1}^\la \Bigl( 1 - x/\la
+ \lfloor x / \la \rfloor \Bigr) + u_{\lfloor x / \la \rfloor + 2} 
\Bigl( x / \la - \lfloor x / \la \rfloor \Bigr), \quad x \geq 0, 
\lambda > 0.
\]
For all $m \geq 0$, the function $f_\lambda$ is linear on each interval $[m \lambda,(m+1) \lambda]$, and satisfies $f_\lambda(m \lambda)=u^{\lambda}_{m+1}$. 
Remark that $\sup_{x \geq 0} \lvert f_{\lambda}(x) \rvert \leq C$ and that
\[ 
\bigl\lvert f_{\lambda}(y) - f_{\lambda}(x) \bigr\rvert \leq C (y- x), 
\quad 0 \leq x < y.
\] 
Since $\varphi(\lambda,t)=\left\lceil\frac{\ln(1-t)}{\ln(1-\lambda)}\right\rceil$, one can easily see that for any $t \in (0,1)$, 
\begin{equation}
\label{Laplace1}
\lim_{\lambda \rightarrow 0+} u^\la_{ \varphi(\lambda, t)} - f_{\lambda} \bigl( - \ln(1 - t) \bigr) 
 = 0.
\end{equation}
Elementary computations also show that for any $\delta > 0$, 
\begin{equation}
\label{Laplace2}
\lim_{\lambda \rightarrow 0+} B_{\lambda}(\delta) - \int_0^{+ \infty}
\delta \exp ( - \delta x) f_{\lambda}(x) \, \mathrm{d} x = 0.
\end{equation}
The continuity properties of the Laplace transform ensure that
$\lim_{\lambda \rightarrow 0+} f_{\lambda}(x) = 0$, $ x > 0$
if and only if 
\[
\lim_{\lambda \rightarrow 0+} \int_0^{ + \infty} \exp ( - \delta x ) f_{\lambda}(x) \, 
\mathrm{d} x = 0, \quad \delta > 0.
\] 
To see this, we can for example apply \cite[XIII.1 Theorem 2. page 431]{Feller}
to the family of probability distributions
\[
Z_{\lambda, \alpha}^{-1} \exp(-\alpha x) \bigl[ f_{\la}(x) + 2 C \bigr] 
\, \mathrm{d} x, \quad \lambda > 0 , \; \alpha > 0, 
\]
on $[0, + \infty)$,
where
\[
Z_{\la, \alpha} = 
\int_0^{+ \infty} \exp(-\alpha x) \bigl[ f_{\la}(x) + 2 C \bigr]
\, \mathrm{d} x. 
\] 
This proves the proposition in view of \eqref{Laplace1} and 
\eqref{Laplace2}.
\end{proof}

\subsubsection{Application to stochastic games}
We now provide several alternative characterisations of the strong constant payoff property which will be used in the proof of Theorem \ref{paiement_constant_SP}. 
\begin{definition}
A family $(\sigma^1_\la,\sigma^2_\la)_\la$ is a \textit{discounted optimal strategy profile} if for all $\lambda \in (0,1]$, $(\sigma^1_\la,\sigma^2_\la)$ is a pair of optimal strategies in $\Gamma_{\la}$.
\end{definition}
\begin{proposition} \label{eqSP1}
Let $(\sigma^1_\la,\sigma^2_\la)_\la$ be a discounted optimal strategy profile. 
The following conditions are equivalent:
\begin{enumerate}
\item[$(i)$] \label{seqconstant}
The family $(\sigma^1_\la,\sigma^2_\la)_\la$ satisfies the constant payoff property for all $k \in K$:
$$\limla \ga^k_\la(\sigma^1_\la,\sigma^2_\la; t)= tv^*(k) \qquad \forall t\in [0,1]\, .$$ 
\item[$(ii)$] \label{tconstant}
For all $k \in K$, for all $t \in [0,1)$, $\m{E}^{k}_{\sigma^1_\la,\si^2_\la}[v_{\lambda}(k_{\varphi(\la,t)})]-v_{\la}(k)$ converges to 0 as $\la$ vanishes. 
\item[$(iii)$] \label{deltaconstant}
For all $k \in K$, for all $\de>0$ one has: 
\begin{equation} \label{carSP}
\limla  \E_{\si^1_\la,\si^2_\la}^{k}\left [\sum\nolimits_{m \geq 1} \de \la (1-\de \la)^{m-1}(v_{\la}(k_m)-v_\la(k))\right]=0.
\end{equation}
\end{enumerate}
\end{proposition}
\begin{proof}

We start by proving the equivalence between $(i)$ and $(ii)$. Fix $k \in K$ and $t\in (0,1)$.
For any $\la\in(0,1]$, Shapley's equation yields
\begin{equation*}
v_{\lambda}(k)=\E_{\si^1_\la,\si^2_\la}^{k}\left [\sum\nolimits_{m=1}^{\varphi(\la,t)-1} \la(1-\la)^{m-1} g_m \right]+(1-\la)^{\varphi(\la,t)-1} \E_{\si^1_\la,\si^2_\la}^{k}\left [v_{\lambda}(k_{\varphi(\la,t)})\right].
\end{equation*}
Consequently,
\begin{eqnarray*}
\E_{\si^1_\la,\si^2_\la}^{k}\left [\sum\nolimits_{m=1}^{\varphi(\la,t)-1} \la(1-\la)^{m-1} g_m \right]-t v_{\lambda}(k)+(1-\la)^{\varphi(\la,t)-1} \left( \E_{\si^1_\la,\si^2_\la}^{k}\left [v_{\lambda}(k_{\varphi(\la,t)})\right]-v_{\lambda}(k)\right)
\\= \left(1-t-(1-\la)^{\varphi(\la,t)-1}\right)v_{\la}(k)\, .
\end{eqnarray*}
The equivalence between $(i)$ and $(ii)$ is obtained by taking $\la$ to $0$, and recalling the relation $\lim_{\la \rightarrow 0} (1-\la)^{\varphi(\la,t)-1}=1-t$. \\[0.2cm]
We now prove the equivalence between $(i)$ and $(iii)$.  
For each $k \in K$, $\lambda \in (0,1]$ and $m \geq 1$, define 
$$u^\la_m:=\E_{\si^1_\la,\si^2_\la}^{k}[v_\la(k_m)]-v_\la(k)\, .$$ 
Note that the family of sequences  $(u^\la_m)$ satisfies $(\ref{condseq})$ with $C=\max_{k,i,j}|g(k,i,j)|$. 
Therefore, Proposition \ref{analyse} applies, and gives the desired result. 
\end{proof}
The following result is now a direct consequence of Proposition \ref{eqSP1}.
\begin{cor} \label{alternative_formulation} Theorem \ref{paiement_constant_SP} holds if and only if for all $(\sigma^1_\la,\sigma^2_\la)_\la$ discounted optimal strategy profile, for all $k \in K$ and $\de>0$, 
\begin{equation} \label{carSP}
\limla  \E_{\si^1_\la,\si^2_\la}^{k}\left [\sum\nolimits_{m \geq 1} \de \la (1-\de \la)^{m-1}(v_{\la}(k_m)-v_\la(k))\right]=0 \,.
\end{equation}
\end{cor}

\subsection{Auxiliary MDP and proof of Theorem \ref{paiement_constant_SP}}


 %


Let $\delta>0$ and ${k} \in K$ be fixed. For each $\lambda \in (0,1]$, consider a Markov Decision Process (one-player stochastic game) $MDP_{{k},\delta,\la}$ with state space $K$, action set $A_\la(\ell):=X^1_{\lambda}(\ell) \times X^2_\lambda(\ell)$ for each $\ell\in K$, transition function $q$, payoff function $\ell \mapsto v_{\la}(\ell)-v_{\la}({k})$ and discount factor $\de\la$.
\begin{remarque} 
At each state, the decision-maker can only play pairs of optimal mixed strategies of the game given in \eqref{shap_game}. Hence, the sets of possible actions depend on the state and on the discount factor. Similarly, the payoff function does not depend on the actions but depends on the discount factor. \end{remarque}

\medskip

\noindent For any pair of optimal strategies $(\sigma^1,\sigma^2) \in \Sigma^1_{\lambda} \times \Sigma^2_{\lambda}$ of the original $\la$-discounted stochastic
game 
and any initial state $\ell\in K$, define 
\begin{equation} \label{discountedMDP}
h_\la(\ell,\si^1,\si^2):= \E_{\si^1,\si^2}^{\ell} 
\left[ \sum\nolimits_{m \geq 1} \de \la (1-\de \la)^{m-1}(v_{\la}(k_m)-v_\la({k}))\right]\, .
\end{equation}
 Let $w_\la^{\de}(\ell)$ denote the value of this MDP with initial state $\ell$, i.e.:
\begin{equation*}
w_\la^{\de}(\ell)=\sup_{(\sigma^1,\sigma^2) \in \Sigma^1_{\lambda} \times \Sigma^2_{\lambda}}  
h_\la(\ell,\si^1,\si^2)\, .
\end{equation*}
\begin{proposition} \label{propMDP}
Theorem \ref{paiement_constant_SP} holds if and only if for all ${k} \in K$ and $\delta>0$ one has $\lim_{\la \rightarrow 0} w_\la^\de({k})=0$. 
\end{proposition}
\begin{proof}
This stems from Corollary \ref{alternative_formulation}. 
\end{proof}
\paragraph{End of the proof of Theorem \ref{paiement_constant_SP}.}
Let $\delta>0$ and ${k} \in K$ be fixed. 
By Proposition \ref{propMDP}, it is enough to prove that $\lim_{\la \rightarrow 0} w_\la^\de({k})=0$. Note that, by Theorem \ref{paiement_constant_WP} and Proposition \ref{eqSP1} $(iii)$, 
we have $\liminf_{\la \rightarrow 0} w_\la^\de({k}) \geq 0$. Thus, it is enough to prove that $\limsup_{\la \rightarrow 0} w_\la^\de({k})=0$. By contradiction, assume that $\limsup_{\la \rightarrow 0} w_\la^\de({k})>\ep$ for some $\ep>0$. \\[0.2cm]
We resort to the semi-algebraic approach. Note that, unlike the classical setup described in Section \ref{semi_alg}, where a two-player zero-sum stochastic game $\Gamma$ was fixed and the discount factor $\la$ was put to 0, here we have a Markov Decision Process $MDP_{{k},\delta,\la}$ (thus, one player only) which depends on $\la$ through its action set and payoff function. Nonetheless, the semi-algebraic approach still applies. 
Define a subset $S_\ep\subset \R \times \R^K\times \R^{K\times I}\times \R^{K\times J} \times \R^K$ by setting
$$(\la, v, x^1,x^2,h)\in S_\ep \quad \text{if, and only if, the following relations hold:} $$
\begin{itemize}
\item $\la\in \R$ is some discount factor, that is $0<\la\leq 1$.
\item $v\in \R^K$ is the vector of values of the $\la$-discounted stochastic game $\Ga_\la$.
\item $(x^1, x^2)\in \R^{K\times I}\times \R^{K\times J}$ is a pair of optimal stationary strategies in $\Gamma_\la$.
\item $h\in \R^K$ satisfies $h(\ell)=h_\la(\ell,x^1,x^2)$ for all $\ell\in K$ and $h({k})\geq \ep$. 
\end{itemize}
The set $S_\ep$ is semi-algebraic, as it can be described by the following finite set of polynomial equalities and inequalities (compare with the system in Section \ref{semi_alg}): \begin{eqnarray*}
0< \la &\leq & 1\\
\forall (\ell,i), \ x^1(\ell,i)\geq 0,\ \text{and } \ \forall \ell, \quad \sum\nolimits_{i\in I} x^1(\ell,i)&=&1\\
\forall (\ell,j), \ x^2(\ell,j)\geq 0,\ \text{and } \ \forall \ell, \quad \sum\nolimits_{j\in J} x^2(\ell,j)&=&1\\
\forall (\ell,j),\quad \sum\nolimits_{i\in I} x^1(\ell,i)\left(\la g(\ell,i,j)+(1-\la)\sum\nolimits_{\ell'\in K}q(\ell'|\ell,i,j)v(\ell')\right) & \geq & v(\ell)\\
\forall (\ell,i),\quad \sum\nolimits_{j\in J} x^2(\ell,j)\left(\la g(\ell,i,j)+(1-\la)\sum\nolimits_{\ell'\in K}q(\ell'|\ell,i,j)v(\ell')\right) & \leq & v(\ell)\\
\forall \ell,\quad \de \la (v_\la(\ell)-v_\la({k})) +(1- \de \la)\sum\nolimits_{(i,j)\in I\times J}x^1(\ell,i)x^2(\ell,j)\sum\nolimits_{\ell'\in K}q(\ell'|\ell,i,j)h(\ell')& = & h(\ell)
\\
h({k}) &\geq &\ep\, .
\end{eqnarray*}

For $\lambda \in (0,1]$, let 
$S_{\ep}(\lambda)=\left\{a \ | \ (\lambda,a) \in S_{\ep}\right\}$. 
By assumption, $\limsup_{\la \rightarrow 0} w_\la^\de({k})>\ep$, thus there exists a vanishing subsequence $(\lambda_n)$ such that for all $n$, the set $S_\ep(\la_n)$
 is non-empty. 
By semi-algebraicity, there exists $\lambda_0 \in (0,1]$ so that $S_\ep(\la)$ is non-empty  for all $\la \in (0,\lambda_0)$. 
From the Tarski-Seidenberg elimination theorem, it admits a semi-algebraic selection for $\la\in (0,\la_0)$. 
In particular, there exists a selection of stationary strategies $z_\la:=(x^1_\la,x^2_\la)$ that is a strategy of the Markov decision process $MDP_{{k},\delta,\la}$, which can be expressed as a Puiseux series near $0$, and so that $h_\la({k},x^1_\la,x^2_\la) \geq \ep$ for all $\la$ small enough. But this contradicts Theorem \ref{paiement_constant_WP} and Proposition \ref{eqSP1} $(iii)$ since, together, they imply that $\lim_{\la \rightarrow 0} h_\la({k},x^1_\la,x^2_\la) =0$. \hfill $\square$

\section{Examples and a remark}\label{section_annexe}
\subsection{An example}
Let us illustrate the constant payoff property by an example, studied by Bewley and Kohlberg \cite{BK78}. 
The state space is the set $K=\{1^*,k,\ell,0^*\}$. For all $(i,j)\in I\times J$, one has $q(1^*\,|\,1^*,i,j)=q(0^*\,|\,0^*,i,j)=1$, 
$g(1^*,i,j)=1$ and $g(0^*,i,j)=0$ so that the states $1^*$ and $0^*$ are absorbing with payoff $1$ and $0$ respectively.
 The transition from states $k$ and $\ell$ are deterministic and represented by the two following matrices: 
\begin{center}
\begin{tikzpicture}[xscale=1, yscale=1]
\draw[fill=black!01] (0,0) rectangle (1.6,1.6);
\draw[thin] (0,0.8)-- (1.6,0.8);
\draw[thin] (0.8,0)-- (.8,1.6);
 \node[scale=0.9] at (0.2+0.2,1.2) {$k$};
 \node[scale=0.9] at (0.2+0.2+0.4,-0.5) {$k$};
 \node[scale=0.9] at (0.2+0.2+0.4+4,-0.5) {$\ell$};
  \node[scale=0.9] at (1+0.2,1.2) {$\ell$};
  \node[scale=0.9] at (0.2+0.2,0.4) {$\ell$};
  \node[scale=0.9] at (1+0.2,0.4) {$1^*$};
 \node [above, scale=0.9] at (0.4,1.6) {L};
 \node [above, scale=0.9] at (1.2,1.6) {R};
 \node [left, scale=0.9] at (0,1.2) {T};
  \node [left, scale=0.9] at (0,0.4) {B};
\draw[fill=black!01] (3+1,0) rectangle (4.6+1,1.6);
\draw[thin] (3+1,0.8)-- (4.6+1,0.8);
\draw[thin] (3.8+1,0)-- (3.8+1,1.6);
  \node[scale=0.9] at (4+1+0.2,1.2) {\textcolor{black}{$k$}};
  \node[scale=0.9] at (3.2+1+0.2,0.4) {\textcolor{black}{$k$}};
    \node[scale=0.9] at (3.2+1+0.2,1.2) {$\ell$};
  \node[scale=0.9] at (4+1+0.2,0.4) {\textcolor{black}{$0^*$}};
 \node [above, scale=0.9] at (3+1.4,1.6) { \textcolor{black}{L}};
 \node [above, scale=0.9] at (4.2+1,1.6) {\textcolor{black}{R}};
 \node [left, scale=0.9] at (3+1,1.2) {\textcolor{black}{T}};
  \node [left, scale=0.9] at (3+1,0.4) {\textcolor{black}{B}};
\end{tikzpicture}
\end{center}
The set of actions are $I=\{T,B\}$ and $J=\{L,R\}$. Finally, the payoff function is given by 
$$ \forall (i,j)\in I \times J, \quad  g(k,i,j)=1\quad \text{ and }\quad g(\ell,i,j)=0\, .$$
Optimal stationary strategies satisfy $x^1(T)=x^2(L)\to_{\la\to 0} 1$ and $x^1_\la(B)=x^2_\la(R)\sim_{\la\to 0} \sqrt \la$, so that  
the induced Markov chain satisfies
\small{$$Q_\la\sim_{\la\to 0}\begin{pmatrix}
1 & 0 & 0 & 0\\
\la & 1& 2\sqrt \la & 0  \\
0  & 2\sqrt \la &1 & \la  \\
0 & 0 & 0 & 1 
\end{pmatrix}.$$}
The limit payoff vector is given by $g^*=(1,1,0,0)\in \R^K$ and, by the symmetry of the game, the  vector of limit values is $v^*=(1,1/2,1/2,0)\in \R^K$.
Let $t\in (0,1)$. 
A direct calculation yields that $Q_\la^{\varphi(\la,t)}$ converges to some $\pi_t$, and $\sum_{m=1}^{\varphi(\la,t)} \lambda(1-\lambda)^{m-1}Q^{m-1}_\la$ converges to some $\Pi_t$, such that
$$\Pi_t=  \begin{pmatrix} t&  0& 0& 0 \\ \frac{t^2}{4} & \frac{2t-t^2}{4} & \frac{2-t^2}{4} & \frac{t^2}{4}\\ 
\rule{0pt}{13pt} \frac{t^2}{4} & \frac{2t-t^2}{4} & \frac{2t-t^2}{4} & \frac{t^2}{4}\\  \rule{0pt}{11pt} 0&  0& 0& t\end{pmatrix},\quad 
\pi_t= \begin{pmatrix} 1&  0& 0& 0 \\ \frac{t}{2} & \frac{1-t}{2} & \frac{1-t}{2} & \frac{t}{2}\\ 
 \rule{0pt}{11pt} \frac{t}{2} & \frac{1-t}{2} & \frac{1-t}{2} & \frac{t}{2}\\  0&  0& 0& 1\end{pmatrix}\,.$$
 In particular,  
$$\Pi= \Pi_1= 
 \begin{pmatrix} 1&  0& 0& 0 \\ \frac{1}{4} & \frac 1 4 & \frac 1 4 & \frac{1}{4} \\ \rule{0pt}{11pt} \frac{1}{4} & \frac 1 4 & \frac 1 4 & \frac{1}{4} \\ \rule{0pt}{10pt} 0&  0& 0& 1\end{pmatrix}.$$
One can thus easily check the equality $\Pi v^*=v^*$, and that $\Pi_t g^*=t v^*$ and $\pi_t v^*=v^*$ hold for all $t\in(0,1)$. 

\subsection{The constant payoff is a joint property}
Contrary to non-zero sum games, where the notion of Nash equilibrium is a joint property of the players' strategies, 
the notion of optimality is unilateral in zero-sum games. Indeed, by playing an optimal strategy, Player $1$ ensures that his payoff is greater than
or equal to 
 the value regardless of the strategy used by his opponent (and similarly for Player 2). Naturally, one would like to know whether the constant payoff property is an unilateral property as well. That is, can Player 1 ensure that the average payoff at time $t$ is greater than or equal to 
 the value at all times $t\in [0,1]$? \\ 

\noindent  The following example gives a negative answer to this question: playing an optimal strategy does not ensure that the average payoffs are greater 
than or equal to 
the value at all times. Rather, the constant payoff property requires \emph{both players} to play optimally. 
The example is ``as bad as it can be'', since the unique optimal strategy of Player $1$ guarantees strictly less than the value at any time $t\in(0,1)$, where the fact that this property holds for $t=1$ follows from the optimality of his strategy.  
\\ 

Consider the ``Big Match'', introduced by Gillette \cite{gillette57},  
a stochastic game with set of states $K=\{k,0^*,1^*\}$, action sets $I=\{T,B\}$ and $J=\{L,R\}$, and where states $0^*$ and $1^*$ are absorbing with payoff $0$ and $1$ respectively, i.e. for all $(i,j)\in I\times J$,
$$q(\,\cdot\,| 0^*,i,j)=\de_{0^*},\quad  q(\,\cdot\,| 1^*,i,j)=\de_{1^*},\quad  g(0^*,i,j)=0\quad \text{ and } \quad   g(1^*,i,j)=1.$$
The game with initial state $k$, the non-absorbing state, can be represented as follows: 
 \begin{center}
\begin{tikzpicture}[xscale=0.7,yscale=0.7]
\draw[fill=black!03]  (0,0) rectangle (2,2);
\draw[thin] (0,1)-- (2,1);
\draw[thin] (1,0)-- (1,2);
 \node at (0.5,0.5) {$0^*$}; 
 \node at (0.5,1.5) {$1^*$};
  \node at (1.5,0.5) {$1^*$}; 
  \node at (1.5,1.5) {$0^*$};
 \node [above] at (0.5,2) {L};
 \node [above] at (1.5,2) {R};
 \node [left] at (0,0.5) {B};
 \node [left] at (0,1.5) {T};
 \end{tikzpicture}
\end{center}
As far as Player $1$ plays action $B$, he receives stage payoffs $0$ or $1$, depending on whether Player $2$ plays $L$ or $R$, and the state does not change, i.e. 
$$g(k,B,L)=0,\quad  g(k,B,R)=1\quad \text{ and }\quad q(k \,|\, k,B,L)=q(k\,|\, k,B,R)=1.$$ When Player $1$ plays $T$, the state moves to an absorbing state, 
 indicated by a $*$ in the picture above, depending on the action of his opponent. 
For all $\la\in(0,1]$, the value and the unique optimal stationary strategy profile are given by
$$v_\la(k)=\frac 1 2,\quad \quad x^1_\la(k,T)= \frac{\la}{1+\la}\quad \text{ and }\quad x^2_\la(k, L)=\frac 1 2.$$ 
Let $x^2$ be the strategy that plays $L$ at every stage. Though not optimal, $x^2$ is a best reply to $x^1_\la$ since 
$\ga_\la(k,x^1_\la,x^2)=v_\la(k)$ for all $\la$. Computations show that 
$$\limla \ga_\la(k,x^1_\la,x^2;t)=\frac{t^2}{2}\,.$$
Since $\frac{t^2}{2}< t v^*(k)$ for all $t\in(0,1)$, $(x^1_\la,x^2)$ does not satisfy the constant payoff property. In fact, under these strategies, Player $2$ obtains strictly less than the value (and this is favorable to him) at all times except for $t=1$. 

\section*{Acknowledgments}
We are greatly indebted to Sylvain Sorin, whose comments have led to significant improvements in the presentation of the paper. We are also very thankful to Abraham Neyman for his careful reading and numerous remarks, and also to Cyril Labbé, Rida Laraki, Eran Shmaya and Guillaume Vigeral for helpful discussions.

\bibliographystyle{amsplain}
\bibliography{bibliothese2}

\end{document}